\documentclass[12pt]{amsart}

\usepackage[english]{babel}
\usepackage{amssymb} 
\usepackage{amsfonts} 
\usepackage{amsmath}
\usepackage{amsthm} 
\usepackage{epsfig} 
\usepackage{color}
\usepackage{pinlabel}
\usepackage{amscd}
\usepackage[all]{xy}
\usepackage{hyperref}
\usepackage{subcaption}
\usepackage{graphicx}
\usepackage{tikz}

\newtheorem{lemma}{Lemma}[section]
\newtheorem{theorem}[lemma]{Theorem}
\newtheorem{theoremx}{Theorem}

\newtheorem{prop}[lemma]{Proposition}

\theoremstyle{definition}

\theoremstyle{remark}
\newtheorem{rem}[lemma]{Remark} 

\newcommand{\Iso}{{\mathrm{Iso}}}

\usepackage[letterpaper,left=1in,right=1in,top=1in,bottom=0.9in]{geometry}
\usepackage{todonotes}

\let\oldbibliography\thebibliography
\renewcommand{\thebibliography}[1]{%
  \oldbibliography{#1}%
  \setlength{\itemsep}{.4\baselineskip}%
}

\author{\tiny{Maria Dostert}}
\address{\tiny{Department of Mathematics, Royal Institute of Technology (KTH), SE--100 44 Stockholm, Sweden}}
\email{\tiny{maria.dostert@gmail.com}}

\author{\tiny{Alexander Kolpakov}}
\address{\tiny{Institut de math\'ematiques, Universit\'e de Neuch\^atel, CH--2000 Neuch\^atel, Switzerland}}
\email{\tiny{kolpakov.alexander@gmail.com}}

\title[Packable hyperbolic surfaces with symmetries]{Packable hyperbolic surfaces with symmetries}

\begin{document}

\begin{abstract}
We discuss several ways of packing a hyperbolic surface with circles (of either varying radii or all being congruent) or horocycles, and note down some observations related to their symmetries (or the absence thereof). \\


\noindent
\textit{2010 AMS Classification: } 05B40, 20H10, 11F06 \\
\end{abstract}

\maketitle

\section{Introduction}\label{intro}

Let $S_g$ be a hyperbolic surface\footnote{Here and below, we always assume that surfaces are connected, complete, closed (i.e. without boundary), and orientable. The Gau\ss--Bonnet theorem implies that a compact surface of genus $g=0$ (sphere) or $g=1$ (torus) cannot be hyperbolic.} of genus $g\geq 2$, i.e. a connected complete compact closed orientable topological surface of genus $g$ endowed with a Riemannian metric of constant sectional curvature $-1$. Let $\Gamma$ be a discrete subgroup of $\mathrm{PSL}_2(\mathbb{R})$ acting by isometries on the upper half--plane model of the hyperbolic plane $\mathbb{H}^2$ with compact quotient: in this case $\Gamma$ is called co--compact. Then the surface $S_g$ can be equivalently described as the quotient $S_g = \mathbb{H}^2 / \Gamma$ for an appropriate co--compact torsion--free $\Gamma < \mathrm{PSL}_2(\mathbb{R})$. 

Let $C$ be a set of geodesic circles embedded in $S_g$ with non--intersecting interiors. Let $T_C$ be the packing graph whose vertices are the centres of the circles in $C$, which are connected by an edge whenever the corresponding circles are tangent. 

The surface $S_g$ is called \textit{packable} by circles in $C$ if its packing graph $T_C$ viewed as embedded in $S_g$ with geodesic edges provides a combinatorial triangulation of $S_g$. Here and below by a (combinatorial) triangulation of $S_g$ we mean $S_g$ as a topological surface with an embedded locally finite graph $T$ such that $S_g \setminus T$ is a union of topological triangles.

The uniformization theorem of Beardon and Stephenson \cite[Theorem 4]{BeardonStephenson}, which is a generalisation of the classical K\"obe--Andreev--Thurston theorem, asserts that for every combinatorial triangulation of a topological genus $g\geq 0$ surface with graph $T$ there exists a constant sectional curvature metric on it (with curvature normalised to $-1$, $0$ or $+1$) and a set of geodesic circles $C$ in this metric such that $T_C$ is isomorphic to $T$.  For more information about packable surfaces, we refer the reader to the monograph \cite{S}. 

The circles in $C$ may have different radii, and the packing that they provide is combinatorially ``tight'': the tangency relation for circles is transitive. This notion has to be contrasted with the notion of a circle packing where all circles are supposed to be congruent (i.e. isometric to each other).

A circle packing on $S_g$ is a set $P$ of congruent radius $r > 0$ geodesic circles embedded in $S_g$ with non--intersecting interiors such that no more radius $r$ circle can be added to it. We shall assume that, in general, $r$ is less than $\mathrm{inj\, rad}\, S_g$, the injectivity radius of $S_g$. 

The density of such a packing $P$ equals the ratio of the area of $S_g$ covered by the circles in $P$ to the total area of $S_g$. Let us recall that the Gau\ss--Bonnet theorem implies $\mathrm{Area}(S_g) $ equals $ 4\pi(g-1)$. This definition agrees with the usual definition of local density if one lifts the circle packing $P$ of $S_g$ to the universal cover $\mathbb{H}^2$. Indeed, once we find a fundamental domain $D$ for the action of $\Gamma$ such that $S_g = \mathbb{H}^2/\Gamma$, the packing $P$ lifts to $\mathbb{H}^2$. Then the local density of $P$ equals the area of $D \cap \{ \gamma(p) \,|\, \gamma \in \Gamma,\, p \in P \}$ divided by the total area of $D$. This definition is independent on the choice of $D$. Moreover, given $\Gamma$ one can classify its fundamental domains which can be useful for computational purposes \cite{LM, LMV}.

In what follows we shall show that in some cases it is easy to see that a surface is packable provided it has enough symmetries. More precisely, let $S_g$ be an orientable genus $g\geq 0$ surface (with a given metric on it), and let $\Iso^+(S_g)$ be the group of orientation--preserving self--isometries of $S_g$.

\begin{theoremx}[Theorem \ref{symmetry-packable}]
Let $S_g$ be a hyperbolic surface of genus $g\geq 2$ such that $|\Iso^+(S_g)| > 12 (g-1)$. Then $S_g$ is packable. 
\end{theoremx}

In the case of a cusped hyperbolic surface $S$ we shall consider its packing by horocycles instead of ordinary geodesic (i.e. compact) circles. Let $\Gamma < \mathrm{PSL}_2(\mathbb{R})$ be a discrete subgroup with finite--area fundamental domain $D$: such $\Gamma$ is called co--finite. As before, we have that $S = \mathbb{H}^2/\Gamma$ for an appropriate co--finite torsion--free subgroup $\Gamma < \mathrm{PSL}_2(\mathbb{R})$. If $P$ is a set of horocycles in $S$, then $P$ can be lifted to $\mathbb{H}^2$ and we can define its local density similar to the case of compact circles by taking the ratio of the area of $D \cap \{ \gamma(p) \,|\, \gamma \in \Gamma,\, p \in P \}$ to the total area of $D$. This again turns out to be the same as the area of the cusps of $S$ determined by the horocycles in $P$ divided by the area of $S$. 
The maximal horocycle packing density in $\mathbb{H}^2$ is known to equal $\frac{3}{\pi}$, cf. \cite{B, Fejes-Toth, Kellerhals}. We show that only a special class of cusped hyperbolic surfaces may achieve maximal packing density.

\begin{theoremx}[Theorem \ref{thm-packable-horocycles}]
Let $S$ be a hyperbolic surface with cusps. Then $S$ is packable by congruent horocycles with packing density $\frac{3}{\pi}$ if and only if $S = \mathbb{H}^2/ \Gamma$ for some $\Gamma < \mathrm{PSL}_2(\mathbb{Z})$, up to an appropriate conjugation in $\mathrm{PSL}_2(\mathbb{R})$.
\end{theoremx} 

Let us recall that one can define three types of ``circles'' in $\mathbb{H}^2$. First, compact circles that are centred at points of $\mathbb{H}^2$. Second, horocycles that are centred at ideal points on $\partial \mathbb{H}^2$. Equivalently, in the upper half--plane model of $\mathbb{H}^2$, horocycles are represented by circles tangent to $\partial \mathbb{H}^2$ and also by horizontal lines. Third, hypercycles  that are centred at hyperideal points or, equivalently, represented by curves equidistant from a geodesic line. 

Here we would like to stress the fact that, in general, circle packings (respectively, horocycle packings) in the hyperbolic plane $\mathbb{H}^2$ behave in a drastically different way, and that their packing density is not necessarily a well--defined quantity \cite{BR}. This difficulty can be alleviated by studying local densities and packings that are invariant under the action of a co--compact (respectively, co--finite) Fuchsian group as discussed above. 

\section*{Acknowledgements}
\noindent
{\small M.D. was partially supported by the Wallenberg AI, Autonomous Systems and Software Program (WASP) funded by the Knut and Alice Wallenberg Foundation. A.K. was supported by the Swiss National Science Foundation, project no.~PP00P2--202667. The authors would like to thank Noam Elkies for drawing their attention to this circle of problems during the AIM workshop ``Discrete geometry and automorphic forms'' in September, 2018. The authors would also like to thank Gareth Jones, Alan Reid and the anonymous referee for numerous comments.}

\section{Compact packable surfaces}\label{compact_packable}

The Teichm\"uller space $\mathcal{T}_g$ contains a dense subset of packable surfaces, for all $g\geq 2$, in the compact case \cite{BowersStephenson}. All cusped hyperbolic surfaces are packable \cite{Williams}. Moreover, by \cite{McCaughan} if $S_g$ is packable then $S_g = \mathbb{H}^2 / \Gamma$, where $\Gamma < \mathrm{PSL}_2(\mathbb{R} \cap \overline{\mathbb{Q}})$, so that $S_g$ is defined over algebraic numbers. Thus being packable puts strong constraints on the metric of $S_g$.

Now, let $O_g$ be an orbifold genus $g$ surface with $k\geq 0$ orbifold points of orders $n_1, \dots, n_k \geq 2$ or, in another words, let $O_g$ be an orbifold of \textit{signature} $(g; n_1, \ldots, n_k)$. Then we say that $O_g$ is packable by a set of circles $C$ if all the orbifold points of $O_g$ are circle centres, and the tangency relation between the circles in $C$ is transitive (like in the manifold case). Thus, if $O'_{g'}$ is an orbifold cover of $O_g$, circles lift always to circles and $O'_{g'}$ is also packable. 

In what follows, $S(p,q,r)$ will be the standard notation for the hyperbolic ``turnover'' orbifold of genus $g = 0$ with three orbifolds points of orders $p, q, r \geq 2$ such that $1/p + 1/q + 1/r < 1$. Such an orbifold can be obtained by identifying two copies of a hyperbolic triangle with angles $\pi/p$, $\pi/q$ and $\pi/r$ along their respective isometric sides. This orbifold is geometrically ``rigid'' meaning that the hyperbolic metric on it is completely determined by the orbifold angles. 

The following is a simple observation that follows from the facts that $S(p,q,r)$ can be packed by three circles, so that any orbifold covering of $S(p,q,r)$ is packable as noted above. 

\begin{lemma}\label{turnover-packable}
Let a hyperbolic surface $S_g$ of genus $g \geq 2$ be a branched covering of $S(p,q,r)$ with $1/p+1/q+1/r < 1$. Then $S_g$ is packable. 
\end{lemma}

Indeed, it is well--known that each $S(p,q,r)$ can be obtained from a hyperbolic triangle $T$ with dihedral angles $\pi/p$, $\pi/q$, $\pi/r$ by making a ``turnover'' (i.e. ``glueing'' two copies of $T$ isometrically along their boundaries or, equivalently, ``doubling'' $T$ along its boundary). Let $a$, $b$, $c$ be the respective side lengths of such a triangle, as shown in Figure~\ref{fig:turnover}. Then the three circular segments of radii $x$, $y$ and $z$ centred at the corresponding vertices becomes circles in $S(p,q,r)$ centred at the orbifold points. Here, we set $x = \frac{a-b+c}{2}$, $y = \frac{a+b-c}{2}$, $z = \frac{-a+b+c}{2}$, and the side lengths $a$, $b$, and $c$ can be determined from the hyperbolic rule of cosines \cite[Theorem~3.5.4]{Ratcliffe}. 

\begin{figure}
\centering
\includegraphics[scale=0.5]{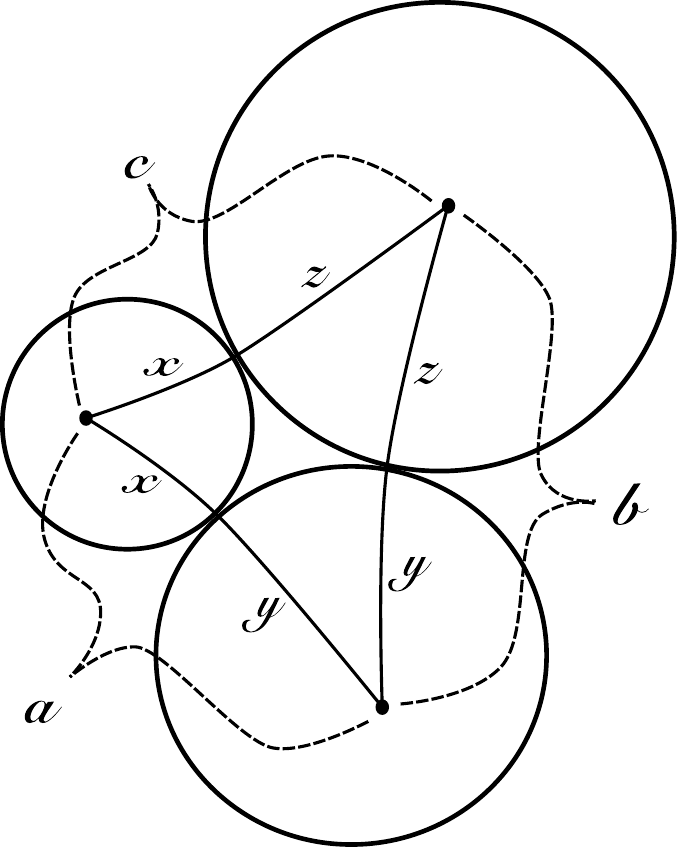}
\caption{A hyperbolic triangle $T$ with side lengths $a$, $b$, $c$, such that the three circles of radii $x$, $y$, $z$ are tangent at points on its sides. Taking the intersection of $T$ with the circles and ``doubling'' it along the boundary produces a packable ``turnover'' orbifold (and its circle packing)}\label{fig:turnover}
\end{figure}

Let $\Iso^+(S_g)$ be the group of orientation--preserving self--isometries of $S_g$, which is known to be finite and isomorphic to $N_H(\Gamma)/\Gamma$, once $S_g = \mathbb{H}^2/\Gamma$, where $N_H(\Gamma)$ is the normaliser of $\Gamma$ in $H = \Iso^+(\mathbb{H}^2) \cong \mathrm{PSL}_2(\mathbb{R})$. 

One can say that ``most'' surfaces in $\mathcal{T}_g$ are asymmetric (i.e. have trivial group of self--isometries), if $g \geq 3$, since such surfaces form an open dense subset of $\mathcal{T}_g$. All genus $2$ surfaces are hyperelliptic, and thus admit an order $2$ isometry: however, most of them (in the above sense) have only their hyperelliptic involution as a non--trivial isometry.  

As mentioned above, most surfaces are not packable, although packable ones form a dense subset in $\mathcal{T}_g$. This motivates the question: ``Given a surface $S_g \in \mathcal{T}_g$ with a certain number of symmetries, can we guarantee that $S_g$ is packable?'' 

We also know, by the Hurwitz automorphism theorem, that $|\Iso^+(S_g)| \leq 84(g-1)$, for any $S_g \in \mathcal{T}_g$ with $g\geq 2$. The next straightforward combinatorial argument shows that given enough symmetries we can always guarantee that $S_g$ covers a turnover orbifold (an observation mentioned earlier in \cite{Singerman}). 

\begin{theorem}\label{symmetry-packable}
Let $S_g$ be a hyperbolic surface of genus $g\geq 2$ such that $|\Iso^+(S_g)| > 12 (g-1)$. Then $S_g$ is packable. 
\end{theorem}
\begin{proof}
We have that $S_g$ is a branched covering of $O = S_g / H$, where $H = \Iso^+(S_g)$. Hence, we can apply the Riemann--Hurwitz formula to it. 

Let us suppose that $O$ is a genus $h\geq 0$ surface with $k\geq 0$ orbifold points of orders $m_1, \dots, m_k$. Thus the ratio $\tau = \mathrm{Area}(O) / (2 \pi)$ satisfies
\begin{equation}
\tau = 2h-2 + \sum^k_{i=1} \left( 1 - \frac{1}{m_i} \right) = \frac{2g-2}{|H|} < \frac{1}{6}. 
\end{equation}

First, if $h \geq 2$, then $\tau \geq 2$, which is impossible by the above inequality. If $h=1$, then in order for $O$ to be an orientable hyperbolic orbifold we need $k\geq 1$. Then, $\tau \geq \frac{1}{2}$. 

Finally, let $h=0$. In this case, the fact that $O$ is an orientable hyperbolic orbifold implies that either we have $k=3$ with $\sum^3_{i=1} \frac{1}{m_i} < 1$, or we have $k=4$ and $m_1, m_2, m_3 \geq 2$, while $m_4 \geq 3$, or $k\geq 5$ with $m_i \geq 2$, $1 \leq i \leq k$. The latter two possibilities give us $\tau \geq \frac{1}{6}$ and $\tau \geq \frac{1}{2}$. 

Thus our case--by--case check implies that $O = S(m_1, m_2, m_3)$, with $\sum^3_{i=1} \frac{1}{m_i} < 1$, and $S_g$ is packable by Lemma~\ref{turnover-packable}.
\end{proof}

\begin{rem}
Note that in Theorem~\ref{symmetry-packable}, we cannot allow $\Iso^+(S_g)$ be $12 (g-1)$ as the following example shows. Let $O$ be an orbifold with signature $(0; 2,2,2,3)$. The corresponding Fuchsian group is known to be maximal \cite{Greenberg}, which means that $O$ does not cover a smaller orbifold. We can map $\pi^{orb}_1(O) = \langle a, b, c, d | a^2, b^2, c^2, d^3, a b c d^{-1} \rangle$ onto the symmetric group $\mathfrak{S}_4$ as $a \mapsto (2,3)$, $b \mapsto (1,2)(3,4)$, $c \mapsto (3,4)$, $d \mapsto (1,2,3)$. Then the kernel of this map $\phi$ is torsion--free and corresponds to a cover $S$ of $O$ of degree $24$. By the Riemann--Hurwitz formula, $S$ is a genus $3$ surface. 

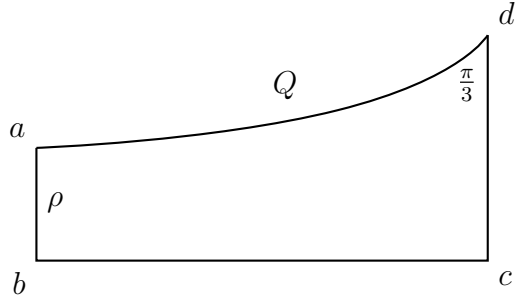
\begin{figure}
\centering
\begin{tikzpicture}[scale = 2]
\draw[black,thick] (0,0.75) -- (0,0) -- (3,0) -- (3,1.5); 
\draw [black, thick] plot [smooth, tension=1] coordinates {(0,0.75)(2,1)(3,1.5)};
\node [above right] at (1.5,1) {$Q$};
\node [below left] at (3,1.35) {$\frac{\pi}{3}$};
\node [above left] at (0,0.75) {$a$};
\node [below left] at (0,0) {$b$};
\node [below right] at (3,0) {$c$};
\node [above right] at (3,1.5) {$d$};
\node [above right] at (0,0.25) {$\rho$};
\end{tikzpicture}
\caption{The Lambert quadrilateral $Q$: all of the plane angles are right, except one angle of $\frac{\pi}{3}$. One side length $\rho$ is known to be variable within a certain interval. The double of $Q$ along its boundary is the orbifold $O$ with signature $(0; 2,2,2,3)$ and fundamental group $\pi^{orb}_1(O) = \langle a, b, c, d | a^2, b^2, c^2, d^3, a b c d^{-1} \rangle$}\label{fig:Lambert}
\end{figure}

Moreover, we can choose $O$ to be non--arithmetic. Indeed, there exists a one--parameter family of orbifolds with signature $(0; 2,2,2,3)$ coming from a  family of Lambert's quadrilaterals $Q$ shown in Figure \ref{fig:Lambert}. A ``double'' of the quadrilateral $Q$ along its boundary is an orbifold $O$ with signature $(0; 2, 2, 2, 3)$. One of its sides can be given length $\rho$ varied so that $\cosh(\rho)$ is a transcendental number. Then $O$ is non--arithmetic by \cite[Theorem 4]{Vinberg} (since the trace of $ab$ equals $2\cosh(\rho)$) and the argument from  \cite[Theorem 1]{Jones-2} shows that $S$ has exactly $|\mathfrak{S}_4| = 24$ symmetries. 

By \cite[Chapter 9]{McCaughan} (cf. also \cite[Theorem 3]{LMS} as a more accessible reference), any packable hyperbolic surface has to be defined as $\mathbb{H}^2/\Gamma$ with $\Gamma < \mathrm{PSL}_2(\mathbb{R} \cap\overline{\mathbb{Q}})$. In our case, however, $\Gamma$ contains the element $r = (ab)^2$ with transcendental trace $2 \cosh(2\rho) = 4 \cosh^2(\rho) - 2 \notin \overline{\mathbb{Q}}$. Since $r \in \ker \phi$, then $S$ cannot be packable.
\end{rem}

However, a packable surface does not need to have any symmetries at all. One should compare the following theorem with the results of \cite{Jones, Jones-2}. 

\begin{theorem}\label{prop:no-aut}
There exists an infinite family of genus $g \geq 14 247$ packable hyperbolic surfaces $S_g$ with $g \to \infty$ such that $\Iso^+(S_g) \cong \{ \mathrm{id} \}$. 
\end{theorem}

For convenience, we shall refer to the following technical lemma, while the proof of Theorem~\ref{prop:no-aut} comes right after. 

\begin{lemma}\label{lemma-split}
All rational primes of the form $p = 3\cdot 5 \cdot 7 \cdot n - 1$, $n \in \mathbb{N}$, split completely in the field $k = \mathbb{Q}(\sqrt{5}, \cos(\pi/7))$.
\end{lemma}
\begin{proof}
Let us consider $k_1 = \mathbb{Q}(\sqrt{5})$, and $k_2 = \mathbb{Q}(\cos(\pi/7))$. Any prime of the form $p = 3 \cdot 5 \cdot 7 \cdot n - 1$ splits completely in both $k_1$ and $k_2$, and thus $p$ splits completely in their compositum, which is $k$ (cf. \cite[Exercise I.8.3]{Neukirch}). 

In $k_1$, we have that $\left( \frac{5}{p} \right) = 1$ by using quadratic reciprocity, and thus $p$ splits completely in $k_1$ by applying \cite[Proposition I.8.3]{Neukirch}.

In $k_2$, we have that $k_2 = \mathbb{Q}(\theta)$ with $\theta = 2\cos(\pi/7)$ that has minimal polynomial $x^3 - x^2 - 2 x + 1$ with discriminant $49 = 7^2$. Thus any prime that is a cubic residue modulo $7$ splits completely in $k_2$ by applying \cite[Proposition I.8.3]{Neukirch} again. 
\end{proof}

\textit{Proof of Theorem \ref{prop:no-aut}.}
Let us consider the non--arithmetic $S(3,5,7)$ orbifold \cite{Takeuchi} with fundamental group  
$\Gamma = \langle a, b \,|\, a^3, b^5, (ab)^7 \rangle$. Moreover, as follows from \cite{Greenberg}, $\Gamma$ is a maximal Fuchsian group.

The trace field $k = \mathbb{Q}(\sqrt{5}, \cos(\pi/7))$ of $\Gamma$ coincides with its invariant trace field. The strong approximation theorem \cite[\S 7.4]{PR} implies that $\Gamma$ surjects on all but a finite number of finite simple groups $\mathrm{PSL}_2(R_k/P)$ where $R_k$ is the ring of integers in $k$ and $P$ a prime ideal in $R_k$.

Note that any non--trivial normal subgroup of $\Gamma$ is  
torsion--free: by using the presentation we easily obtain that any homomorphism $\phi$ with $\phi(a)$ or $\phi(b)$ trivial has trivial image.

By Dirichlet's theorem on prime progressions, there are infinitely many  
rational primes of the form $p = 3 \cdot 5 \cdot 7 \cdot n - 1$. Moreover, such primes $p$ all split completely in $k$ by Lemma~\ref{lemma-split}:
$$ p R_k = P_1 \cdot P_2 \cdot \ldots \cdot P_6,$$
with $P_i$ being distinct prime ideals of $R_k$ with norms $N(P_i) = p$, $i = 1, \ldots, 6$. Thus $R/P_i = Z/p =F_p$, $i = 1, \ldots, 6$.

Consider the epimorphisms $\phi_p: \Gamma \rightarrow \mathrm{PSL}_2(F_p)$ where $p$ is a rational prime of the form $p = 105 \cdot n - 1$, as above. From the description of the subgroup structure of $\mathrm{PSL}_2(F_p)$ from \cite[Theorem 2.1]{King} (cf. also \cite{Dickson}), it follows that $\mathrm{PSL}_2(F_p)$ contains a dihedral group $D_{p-1}$ of order $p - 1 = 105 \cdot n - 2$.

From the order of $D_{p-1}$, we have that it does not contain any elements of order $3$, $5$, or $7$. By the above discussion of normal subgroups of $\Gamma$, the kernel of $\phi_p$ is torsion--free. Hence the preimage $\Gamma_p$ of $D_{p-1}$ is torsion--free.  Moreover, $D_{p-1}$ is a maximal subgroup of $\mathrm{PSL}_2(F_p)$ \cite[Corollary 2.2]{King} (cf. \cite{Dickson}), and so $\Gamma_p$ is a maximal non--normal subgroup of $\Gamma$.

Let $S_p = \mathbb{H}^2/\Gamma_p$. Then $S_p$  is an asymmetric surface, since  $\mathrm{Iso}^+(S_p) = N_{\mathrm{PSL}_2(\mathbb{R})}(\Gamma_p) / \Gamma_p = N_\Gamma(\Gamma_p) / \Gamma_p = \{ \mathrm{id} \}$. From the previous discussion, $S_p$ is clearly packable. 

As $D_{p-1}$ has index $p(p+1)/2$ in $\mathrm{PSL}_2(F_p)$ by \cite[Theorem 2.1(d)]{King} (cf. \cite{Dickson}), the genus of $S_p$ grows quadratically with $p$ and its exact value can be computed by using the Riemann--Hurwitz formula:
$$ g(S_p) = \frac{p(p+1)}{4} \cdot \left(  1 - \frac{1}{3} - \frac{1}{5} - \frac{1}{7} \right) + 1 = \frac{17}{210} p (p+1) + 1.$$

The smallest prime $p = 105\cdot n - 1$ is $419$, which corresponds to $g(S_p) = 14247$. 
\qed

\begin{rem}
By using \texttt{GAP} \cite{GAP}, among the $335$ different genus $3$ surfaces that cover the non--arithmetic orbifold $S = S(2,6,9)$, we find $254$ asymmetric ones. In order to obtain these numbers, we need to use \texttt{LowIndexSubgroupsFpGroup} routine to classify all index $18$ subgroups of $\pi^{orb}_1(S) = \langle a, b \, | \, a^2, b^6, (ab)^9 \rangle$ without torsion, and then choose those that correspond to conjugacy classes of length $18$. These are self--normalising, and by the same argument as in Theorem~\ref{prop:no-aut}, we get the smallest examples of packable surfaces without non--trivial isometries, with all genus $2$ surfaces being hyperelliptic.  
\end{rem}

\section{Packing density on surfaces}

In this section we shall consider circle packings of surfaces with congruent circles. More precisely, let $S_g$ be a hyperbolic surface of genus $g\geq 2$ with a packing of $K$ congruent radius $r$ circles on it, where $0 < r < \mathrm{inj\, rad}\, S_g$. Then the area covered by the circles is $2\pi K (\cosh r - 1)$, while the total surface area equals $4\pi(g-1)$. Then the packing density is simply $\rho(S_g, r) = \frac{K}{2} \cdot \frac{\cosh r - 1}{g - 1}$. The largest packing density of radius $r > 0$ circles associated with a given genus $g\geq 2$ is defined as $\rho(g, r) = \sup_{S_g \in \mathcal{T}_{g,r}} \rho(S_g, r)$, where $\mathcal{T}_{g,r} = \mathcal{T}_g \cap \{ S_g \, | \, \mathrm{inj\, rad}\, S_g > r \}$. By convention, supremum over the empty set equals $0$. Then the packing density associated solely with the genus $g \geq 2$ is $\rho(g) = \sup_{r>0} \rho(g ,r)$. 

\begin{prop}\label{prop-packing-density}
The following limit identity takes place: $\limsup_{g \to \infty} \rho(g) = \frac{3}{\pi}$.
\end{prop}

\begin{proof}
The idea is to pick an orbifold $\Sigma = S(p,p,p)$ such that $p \geq 4$ is an arbitrarily large natural number, and construct its manifold cover $S_p$. Since the orbifold fundamental group $\pi^{orb}_1(\Sigma) \cong \langle a, b \, | \, a^p, b^p, (ab)^p \rangle$ is a finitely generated matrix group then, by Selberg's lemma, such a cover $S_p$ always exists, and its degree has to be at least $p$ by a simple observation about the order of torsion elements. Thus the genus $g_p$ of $S_p$ satisfies $g_p \geq \mathrm{const} \cdot p$. 

Then we obtain a surface $S_p = \mathbb{H}^2 / \Gamma$ of genus $g_p$ such that $g_p \rightarrow \infty$, and $S_p$ is packable by a set $C_p$ radius $r_p$ circles, such that $\cosh r_p = \frac{1}{2}\, \csc \frac{\pi}{2p}$. What remains is to compute the ratio of the area covered  by circles on $S_p$ to the area of $S_p$. As it follows easily from the covering argument, this is the ratio of three $\frac{1}{2p}$--th pieces of a radius $r_p$ disc to the area of an equilateral triangle with angles $\frac{\pi}{p}$. 

The area in $\mathbb{H}^2$ enclosed by a radius $r>0$ circle equals $2\pi (\cosh r - 1)$. Thus we obtain 
\begin{equation}
\rho(S_p, C_p) = \frac{\frac{3\pi}{p} (\cosh r_p - 1)}{\pi - 3\pi/p} = \frac{3}{p-3} \left( \frac{1}{2} \csc \frac{\pi}{2p} - 1 \right) \rightarrow \frac{3}{\pi}, 
\end{equation} 
as $p \to \infty$, by using the Laurent $\csc(x) = \frac{1}{x} + \frac{x}{6} + O(x^2)$ for $\csc(x)$ at $x=0$.

Moreover, all circles and their Voronoi domains in the packing of $S_p$ are congruent to each other (since the order $3$ central symmetry of $\Sigma$ lifts to $S_p$ as the latter is a normal cover). The same holds for the lift of the packing to $\mathbb{H}^2$. The local density of each circle in its Voronoi domain is then equal to $\rho(S_p, C_p)$. The former can be estimated above by the simplicial packing density $d(2r)$, associated with a regular triangle of side length $2r$. However, the latter is exactly what we already computed above, i.e. simply $\rho(S_p, C_p) = d(2r) \leq d(\infty) = \frac{3}{\pi}$, as stated in \cite{BF, Kellerhals}. The proposition follows.
\end{proof}

\begin{rem}
Numerically, in Proposition~\ref{prop-packing-density} we have $\rho(g) \approx 0.954929658551$ for large enough genus $g$, which means that some of the genus $g$ surfaces may be very densely packed. However, some other ones across $\mathcal{T}_g$ may be packed quite poorly. 
\end{rem}

\begin{rem}
It is also clear from Proposition~\ref{prop-packing-density} that the best packing density in $\mathbb{H}^2$ achieved by invariant circle packings (i.e. circle packings invariant under the action of a co--finite Fuchsian group) coincides with the best local packing density achieved by packing congruent horoballs in the ideal triangle \cite{BF, Kellerhals}. 
\end{rem}

\section{Surfaces with cusps}
As a generalisation of the above facts to the case of hyperbolic surfaces with cusps, let us consider packing by horocycles instead of compact circles. In this case, the above question about a surface being packable can be restated without much alteration. The paper \cite{Williams} shows that each cusped hyperbolic surface is known to be packable by circles, while horocycle packings seem to be less studied in this context. 

\begin{figure}[h]
\centering
\includegraphics[scale=0.5]{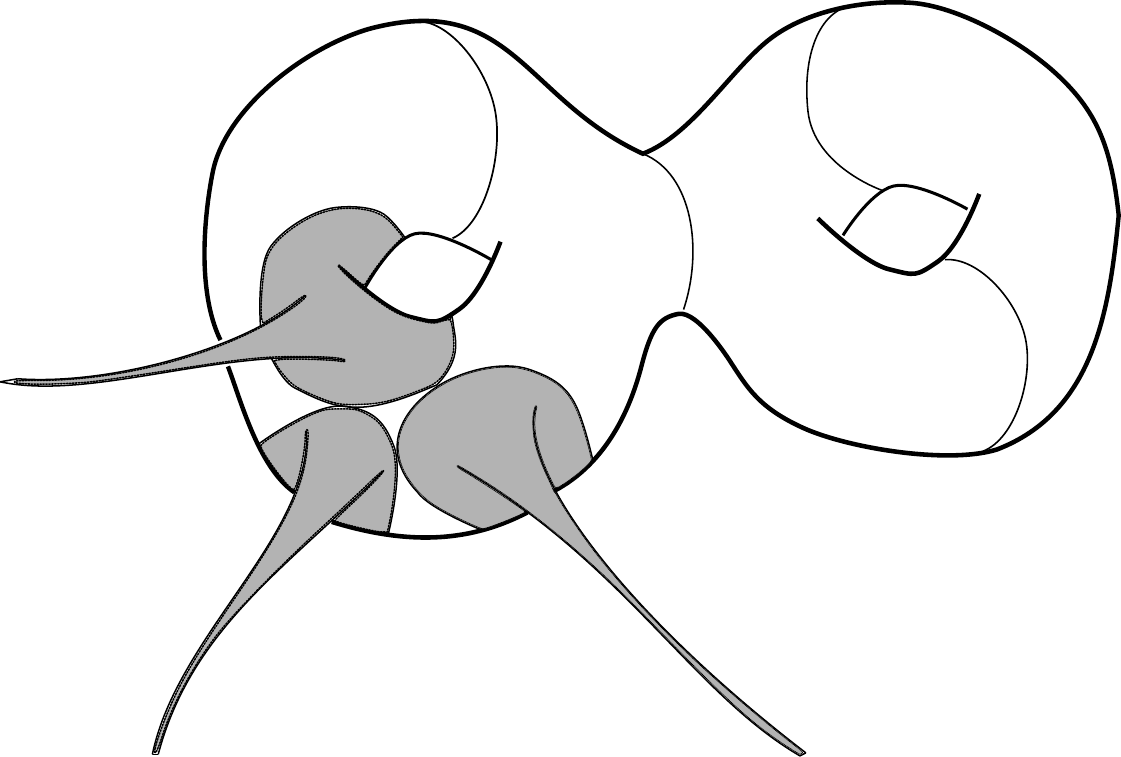}
\caption{A surface with $3$ cusps at which the corresponding horoballs (shaded) are mutually tangent. By closing up the cusps (which are topologically punctures) with extra points $c_i$, $i=1,2,3$, we compactify the surface and obtain the packing graph on it with vertices exactly $c_i$, $i=1,2,3$. The surface depicted above appears to be non--packable}\label{fig:horocusp}
\end{figure}

In the context of horocycle packings of a cusped surface $S$, we suppose that all horocycles are centred in the cusps of $S$, as shown in Figure~\ref{fig:horocusp}. The packing graph $T$ is formed by taking the completion $\overline{S}$ of $S$, with cusps ``filled'' by adding a number of points $c_1, \ldots, c_k$ (where $k \geq 1$ is the number of cusps of $S$), and then letting the vertices of $T$ be exactly $c_i$'s, while two vertices are connected by an edge whenever the corresponding horocycles are tangent. Then, we say that $S$ is packable by horocycles if $T$ triangulates $\overline{S}$, i.e. $\overline{S} \setminus T$ is a collection of topological triangles. 

\begin{theorem}\label{thm-packable-horocycles}
Let $S$ be a hyperbolic surface with cusps. Then $S$ is packable by congruent horocycles with packing density $\frac{3}{\pi}$ if and only if $S = \mathbb{H}^2/ \Gamma$ for some $\Gamma < \mathrm{PSL}_2(\mathbb{Z})$, up to an appropriate conjugation in $\mathrm{PSL}_2(\mathbb{R})$.
\end{theorem} 

\begin{proof}
Let $S$ be packed by congruent horoballs centred at its cusps, so that the packing graph $T$ triangulates $\overline{S}$, then each triangle $T_m$, $m=1, \ldots, k$, in $\overline{S}\setminus T$ is an ideal triangle that contains some parts of horocycles in its vertices. Let $\rho_m$ be the local density of the horocycles in $T_m$, i.e. the ratio of the area contained in the horocycles to the total area of $T_m$. It is well--known from hyperbolic geometry that all ideal triangles are isometric and have area $\pi$. 

Another important fact is that the best local density of horoballs in the ideal triangle $T_m$, which is $\frac{3}{\pi}$, is achieved by $3$ congruent horoballs bounded by the respective horocycles $C^m_k$, $k = 1,2,3$, centred at the vertices of $T_m$. Each pair $C^m_i$ and $C^m_j$ is tangent at a point $p^m_{ij}$ on the respective side of $T$. The perpendiculars to the sides at $p^m_{ij}$'s intersect in the common point inside $T_m$, which we shall call its centre $O_m$, as shown in Figure~\ref{fig:max-configuration}. Let us call such a configuration of horoballs in $T_m$ \textit{the maximal configuration}, which is known to be unique. The density and uniqueness of the maximal configuration are discussed in \cite[Theorem 4]{B}, \cite[VIII.38, pp.~253--254]{Fejes-Toth}, and \cite[Proposition 2.2]{Kellerhals} (see also the remarks thereafter).

\begin{figure}[ht]
\centering
\includegraphics[scale=0.9]{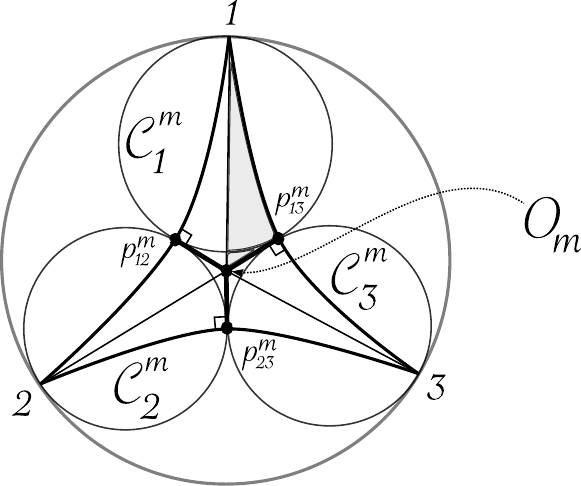}
\caption{An ideal triangle $T_m$ with its maximal horoball configuration. It splits into $6$ triangles with $0$, $\pi/3$ and $\pi/2$ angles centred around $O_m$ (one of the triangles is shaded)}\label{fig:max-configuration}
\end{figure}

Let $\rho(C)$ be the density of the horocycle packing $C = \{ C_1, C_2, \ldots, C_k \}$ of $S$. Then, by assumption, we have that 
$\rho(C) = \frac{1}{k} \sum^k_{m=1} \rho_m = \frac{3}{\pi}$, while $\rho_m \leq \frac{3}{\pi}$ by the above bound on horoball density in ideal triangles. This implies that $\rho_m = \frac{3}{\pi}$, and each triangle $T_m$ has the maximal density configuration of horoballs. Thus we can drop $3$ perpendiculars from the centre $O_m$ of $T_m$ onto its sides, and split $T_m$ into $6$ congruent hyperbolic triangles $\Delta$ with dihedral angles $0$, $\pi/3$ and $\pi/2$. Since the surface $S$ becomes tessellated by copies of $\Delta$ such that each copy can be obtained from one of the neighbours by reflecting in one of the sides, we obtain that $S$ covers the reflection orbifold $O_\Delta = \mathbb{H}^2/W(\Delta)$, where $W(\Delta)$ is the Coxeter group generated by reflections in the lines supporting the sides of $\Delta$. Since $S$ is orientable, then $S \rightarrow O_\Delta$ factors through the orientation cover,  and we obtain $S = \mathbb{H}^2/\Gamma \rightarrow O^+_\Delta$, where $O^+_\Delta$ is the orientation cover of $O_\Delta$ and $\pi^{orb}(O^+_\Delta) = \langle a, b, c \,|\, abc, b^2, c^3 \rangle \cong \mathrm{PSL}_2(\mathbb{Z})$.  Thus $\Gamma < \mathrm{PSL}_2(\mathbb{Z})$, up to conjugation in $\Iso^+(\mathbb{H}^2) = \mathrm{PSL}_2(\mathbb{R})$.

If, up to conjugation, $\Gamma < \mathrm{PSL}_2(\mathbb{Z})$, then $S$ covers the reflection orbifold $\mathbb{H}^2 / W(\Delta)$, where $\Delta$ is the triangle with $0$, $\pi/3$, $\pi/2$ angles as above, and $W(\Delta)$ is its reflection group. We take the maximal horoball bounded by the horocycle $C$ centred at the ideal vertex of $\Delta$ such that $C$ is tangent to the opposite side (exactly at the right--angled vertex). Then the local density of $C$ in $\Delta$ is $\frac{3}{\pi}$ (by a straightforward computation similar to that of Proposition~\ref{prop-packing-density}). Thus $C \cap \Delta$ is lifted to a horocycle packing on $S$ with packing density exactly $\frac{3}{\pi}$.
\end{proof}

\begin{rem}
A more general case of horoball packing where horoballs are allowed to have \textit{different types} at different vertices of the respective Coxeter tilings are considered in \cite{Sz1, Sz2}. 
\end{rem}

\end{document}